\documentclass{article}

\usepackage{arxiv}

\usepackage[utf8]{inputenc} 
\usepackage[T1]{fontenc}    
\usepackage{hyperref}       
\usepackage{url}            
\usepackage{booktabs}       
\usepackage{amsfonts}       
\usepackage{nicefrac}       
\usepackage{microtype}      
\usepackage{lipsum}		
\usepackage{graphicx}
\usepackage{natbib}
\usepackage{doi}
\usepackage{amsmath,amssymb,amsthm}
\usepackage{geometry}
\usepackage{hyperref}
\usepackage{subcaption}
\usepackage{color}
\usepackage{float}

\newtheorem{assumption}{Assumption}[section]
\newtheorem{theorem}{Theorem}[section]
\newtheorem{lemma}[theorem]{Lemma}
\newtheorem{example}[theorem]{Example}

\title{Invariant measures of the stochastic theta method for stochastic differential equations with super-linearly growing coefficients}

\author{Xiaotong Li\\
	Department of Mathematics\\
	Jiangsu Second Normal University\\
	Nanjing, 210013, China \\
	\texttt{x.t.li@foxmail.com} \\
	\And
	Wei Liu\\
	Department of Mathematics\\
	Shanghai Normal University\\
	Shanghai, 200234, China \\
	\texttt{weiliu@shnu.edu.cn} \\
	\And
	Wenjie Xiao\thanks{Corresponding author} \\
	Department of Mathematics\\
	Shanghai Normal University\\
	Shanghai, 200234, China \\
	\texttt{wenjie20020618@163.com} \\
}

\renewcommand{\shorttitle}{Invariant measures of the SDE for stochastic theta method}

\hypersetup{
pdftitle={A template for the arxiv style},
pdfsubject={q-bio.NC, q-bio.QM},
pdfauthor={David S.~Hippocampus, Elias D.~Striatum},
pdfkeywords={First keyword, Second keyword, More},
}

\begin{document}
\maketitle

\begin{abstract}
The stochastic theta method is proposed to approximate invariant measures of stochastic differential equations (SDEs), both of whose drift and diffusion coefficients may grow super-linearly. For the numerical solution generated by the stochastic theta method, we show the existence and uniqueness of the numerical invariant measure first. Then, we prove that the numerical invariant measure is convergent to the exact invariant measure of the underlying SDE. We also provide some numerical simulations to illustrate our theoretical results. This work could be regarded as an extension of the results in [Y. Jiang et al, Numer. Algorithms 83(4)(2020), pp. 1531–1553] to the case of super-linearly growing diffusion coefficient. As the backward Euler-Maruyama (EM) method is a special case of the stochastic theta method, the results derived in this work could also be regarded as a generalization of the results for the backward EM method in [W. Liu et al. Appl. Numer. Math. 184(2023), pp. 137–150] to the stochastic theta method.
\end{abstract}


\section{Introduction}\label{sec:intro}
Invariant measures of stochastic differential equations (SDEs) have many applications. For classical population models, invariant measures is related the persistency of species \cite{Allen2007,Allen2015}. For generative diffusion model, the invariant measure serves as the stable terminal distribution for the forward noising process and provides a well‑defined target distribution for the reverse generative process, ensuring theoretical stability and sampling consistency \cite{CaoTanGaoXuChenHangLi2024,Songetal2021}. Although invariant measures are of importance in many areas, their explicit forms are hardly obtained. Therefore, it is essential to select appropriate numerical methods to approximate the invariant measure and conduct rigorous numerical analysis. In this paper, we employ the stochastic theta method and explore its capability to approximate the invariant measure of SDEs with super-linearly growing coefficients.

\par
The stochastic theta method is a set of numerical methods for SDEs, since the classical Euler-Maruyama (EM) method can be derived by setting the theta to be zero and the backward EM method can be obtained by fixing the theta to be one \cite{Higham2000}.
\par
As one of the most popular numerical methods, the stochastic theta method has been broadly applied and studied. In the following, we first try to review existing literature about the stochastic theta method from different aspects. Then, the purpose of this paper and our main contributions are introduced.
\par
In the past several decades, the stochastic theta method has been applied to different types of SDEs. Higham et al. explored it for SDEs with Markovian switching \cite{HighamMaoYuan2007}. Huang applied it to stochastic delay differential equations \cite{Huang2014}. Li and Gan employed it for stochastic delay integro-differential equations \cite{LiGan2012}. Liu et al. studied it for delayed stochastic Hopfield neural networks \cite{LiuDengZhu2018}. Chen et al. investigated it for a class of SDEs with random periodic solutions \cite{ChenCaoChen2025}. Niu et al. discussed it for free stochastic differential equations \cite{NiuWeiYinZeng2025}. Yuan and Zhu used it for stochastic McKean-Vlasov equations \cite{YuanZhu2026}.
\par
When noises other than the standard Brownian motion are used to drive SDEs, the stochastic theta method still has its place in numerical approximations. Li et al. applied it to a class of SDEs driven by fractional Brownian motion \cite{LiHuHuangWang2023}. Wang et al. employed it for some SDEs driven by Poisson jump process \cite{WangChenNiuNiu2023}. Chen et al. studied it for SDEs driven by time-changed L\'evy noise \cite{ChenLiuWu2026}.
\par
The stochastic theta method is an implicit method, if the theta is not taken to be zero. One of the advantages that implicit methods naturally possess is that they are capable to handle SDEs with super-linear coefficients. For the topic of the finite time strong convergence, Wang et al. proved the mean-square convergence rate of the stochastic theta method for SDEs under a coupled monotonicity condition \cite{WangWuDong2020}. For the topic of the stability of the trivial solution, Zong and Wu studied the mean-square exponential stability of the stochastic theta method and revealed effects of different values of the theta on handling the super-linear coefficients. For the topic of the approximations to the exact invariant measures of SDEs, Jiang et al. used the stochastic theta method to numerically approximate the invariant measure of some SDE whose drift coefficient grows super-linearly but diffusion coefficient still satisfies the linear growth condition \cite{JiangWengLiu2020}.
\par
In this paper, we work a bit harder and manage to show that the numerical invariant measure generated by the stochastic theta method is still convergent to the exact one even when the diffusion coefficient grows super-linearly. So, the results obtained in this work could be regarded as non-trivial extension of those results in \cite{JiangWengLiu2020} to SDEs with both the drift and diffusion growing super-linearly. As mentioned above that the backward EM method is a special case of the stochastic theta method with the theta being one, the results obtained in this work could also be regarded as a generalization of the results of the backward EM method in \cite{LiuMaoWu2023}. It should be mentioned that techniques neither in \cite{JiangWengLiu2020} or \cite{LiuMaoWu2023} could be straightly adapted in this work to obtain our results. We need to carefully deal with the remained explicit part in the stochastic theta method when the theta taking values less than one. 
\par
Before we close Section \ref{sec:intro}, it should be noted that various other numerical methods, apart from the stochastic theta method, were also employed to approximate the exact invariant measures of SDEs. We just mention some of them here \cite{BaoShaoYuan2016,LiMaoYin2019,LiuLiu2025,PangWangWu2024} and refer the readers to the references therein. 
\par
The rest of the paper is organized as follows. In section \ref{sec:preliminary}, the mathematical preparations are presented and the stochastic theta method is revisited. The main results are stated and proved in Section \ref{sec:results}. The numerical examples are provided in Section \ref{sec:numexp}. Conclusions, discussions and future research are put in Section \ref{sec:conclusion}.

\section{Mathematical preliminary}\label{sec:preliminary}
In this section, we give some basic settings, notations and useful lemma for the rest of the paper.

Throughout this paper, we work with the complete probability space $(\Omega, \mathcal{F}, \mathbb{P})$. Let $|a|$ denote the Euclidean norm if $a$ is a vector and the Hilbert-Schmidt norm if $a$ is a matrix. The inner product of vectors $a$ and $b$ in the Euclidean space is denoted by $\langle a,b\rangle$. The lager one between $a$ and $b$ is denoted by $a\vee b$, and the smaller one is denoted by $a\wedge b$. For a vector or matrix $a$, $a^{\mathrm{T}}$ denotes its transpose. Let $\mathbb{I}_A$ be the indicator function for set $A$. The family of all probability measures on
$\mathbb{R}^n$ is denoted by $\mathcal{P}(\mathbb{R}^n)$. Let $\mathcal{B}(\mathbb{R}^n)$ denote the family of all Borel sets in $\mathbb{R}^n$. The $p$-Wasserstein distance between $u,v\in\mathcal{P}(\mathbb{R}^n)$ for any $p\in(0,1)$ is defined by 
\begin{equation*}
{W}_p(\mu_1,\mu_2)=\inf_{\nu\in\mathcal{C}(\mu_1,\mu_2)}\int_{\mathbb{R}^n\times\mathbb{R}^n}|X-Y|^p\nu(dX,dY),
\end{equation*}
where $\mathcal{C}(\mu_1,\mu_2)$ denotes the set of all couplings of $\mu_1$ and $\mu_2$.

Given a stochastic process $X(t)$ on $(\Omega, \mathcal{F}, \mathbb{P})$, denote the transition probability kernel of $X(t)$ by $\mathbb{P}_t(x,B)$ for any $t>0$ and any $B\in\mathcal{B}$. We use $\delta_x$ to emphasise the initial value $x$ and denote $\mathbb{P}_t(x,\cdot)$ by $\delta_x\mathbb{P}_t$. A probability measure $\pi(\cdot)\in\mathcal{P}(\mathbb{R}^n)$ is called an invariant measure of $X(t)$, if 
\begin{equation*}
\pi(B)=\int_{\mathbb{R}^n}\mathbb{P}_t(x,B)\pi(dx)
\end{equation*}
holds for any $t>0$ and any $B\in\mathcal{B}(\mathbb{R}^n)$.

Let $W(t)$ be a $d$-dimensional Brownian motion defined on $(\Omega,\mathcal{F},\mathbb{P})$. We consider the $n$-dimensional SDE of the It\^o type
\begin{equation}\label{eq:SDE}
dX(t)=b(X(t))dt+g(X(t))dW(t)
\end{equation}
with initial value $X(0)=x_0\in\mathbb{R}^n$, where $b:\mathbb{R}^n\rightarrow \mathbb{R}^n$ and $g:\mathbb{R}^n\rightarrow\mathbb{R}^{n\times d}$. We impose several assumptions on $b$ and $g$ as follows.
\begin{assumption}\label{assum:2.1}
There exists a pair of constant $L>0$ and $q>1$ such that 
\begin{equation*}
	|b(x_1)-b(x_2)|\vee |g(x_1)-g(x_2)|\leq L(1+|x_1|^{q-1}+|x_2|^{q-1})|x_1-x_2|
\end{equation*}
for all $x_1,x_2\in\mathbb{R}^d$.
\end{assumption}
\begin{assumption}\label{assum:2.3}
There exist constants $l_1\geq 1$ and $\beta > 0$ such that
\[2\langle x_1-x_2, b(x_1)-b(x_2)\rangle +l_1|g(x_1)-g(x_2)|^2 \le -\beta|x_1-x_2|^2
\]
for all $x_1,x_2\in\mathbb{R}^d$.
\end{assumption}
\begin{assumption}\label{assum:2.2}
There exist constants $l_2\geq 1$,$c_1 > 0$ and $\alpha > 0$ such that
\[2\langle x, b(x)\rangle + l_2|g(x)|^2 \le c_1 -\alpha|x|^2
\]
for all $x \in \mathbb{R}^d$.
\end{assumption}

Assumption \ref{assum:2.1} together with Assumption \ref{assum:2.3} guarantees the existence and uniqueness of the solution to \eqref{eq:SDE}, see, e.g., \cite[Theorem~3.4]{Mao2007}.

For a given time step $h\in(0,1)$, the stochastic theta method to SDE \eqref{eq:SDE} is defined by 
\begin{equation}\label{eq:NumSDE}
X_{k+1}=X_k+\theta b(X_{k+1})h+(1-\theta)b(X_k)h+g(X_k)\Delta W_k, \quad X_0=X(0)=x_0,
\end{equation}
where $\theta\in(1/2,1]$ and $\Delta W_k=W(t_{k+1})-W(t_k)$ is the Brownian motion increment with $t_k=kh$, for $k=1,2,\cdots$.

To investigate the invariant measure of the numerical solution, we introduce the following notation. For any $k\in\mathbb{N}$ and $B\in\mathcal{B}(\mathbb{R}^n)$, let $\hat{\mathbb{P}}_k^h(x,B)$ be the transition probability kernel of $
X_k$. A probability measure $\hat{\pi}(\cdot)\in\mathcal{P}(\mathbb{R}^n)$ is called an invariant measure of $X_k$, if 
\begin{equation*}
\hat{\pi}(B)=\int_{\mathbb{R}^n}\hat{\mathbb{P}}_k^h(x,B)\hat{\pi}(dx)
\end{equation*}
holds for any $k\in\mathbb{N}$ and $B\in\mathcal{B}(\mathbb{R}^n)$.

The implementation of \eqref{eq:NumSDE} requires solving a nonlinear equation at each iteration. The well-posedness of the difference equation \eqref{eq:NumSDE} is proved in the next lemma. The proof is analogous to that of \cite{JiangWengLiu2020} and is therefore omitted.

\begin{lemma}
Let Assumption \ref{assum:2.3} holds. For any $h\in(0,1)$ and $\theta\in(1/2,1]$, then the stochastic theta method \eqref{eq:NumSDE} is well defined.
\end{lemma}

\section{Main results}\label{sec:results}
In this section, the existence and uniqueness of the invariant measure of the numerical solution generated by the stochastic theta method is proved. Furthermore, the convergence of the numerical invariant measure to the underlying counterpart is discussed.

To obtain the existence and uniqueness of the invariant measure, we need the following two auxiliary lemmas.

\begin{lemma}\label{lem:3.1}
Let Assumption \ref{assum:2.2} holds. Define 
\[
\rho_0 = \left\{ \dfrac{\alpha}{2}\wedge \dfrac{2\theta - 1}{2\theta^2} \right\}.
\]
For any $h\in(0,1)$ satisfying $0 < h < 1/\rho_0$, the numerical solution \eqref{eq:NumSDE} obeys
\[
\mathbb{E}|X_k|^2 \le C_1,
\]
where $C_1 > 0$ is a constant.
\end{lemma}
\begin{proof}
Define $ M_k = X_k - \theta h b(X_k) $,  
\[
\begin{aligned}
	M_{k+1} &= X_{k+1} - \theta h b(X_{k+1}) \\
	&= X_k + (1+\theta)h b(X_k) + g(X_k)\Delta W_k - \theta h b(X_{k+1}) \\
	&= X_k + h b(X_k) - \theta h b(X_k) + \theta h b(X_k) + g(X_k)\Delta W_k - \theta h b(X_{k+1}) \\
	&= M_k + h b(X_k) + g(X_k)\Delta W_k.
\end{aligned}
\]  
Squaring both sides yields  
\[
\begin{aligned}
	|M_{k+1}|^2 &= |M_k + h b(X_k) + g(X_k)\Delta W_k|^2 \\
	&= |M_k|^2 + h^2 |b(X_k)|^2 + |g(X_k)\Delta W_k|^2 + 2\langle M_k, h b(X_k) \rangle \\ &\quad + 2\langle M_k, g(X_k)\Delta W_k \rangle + 2\langle h b(X_k), g(X_k)\Delta W_k \rangle.
\end{aligned}
\]  
Taking expectation,  
\[
\mathbb{E}|M_{k+1}|^2 = \mathbb{E}|M_k|^2 + h^2 \mathbb{E}|b(X_k)|^2 + h \mathbb{E}|g(X_k)|^2 + 2h \mathbb{E}\langle M_k, b(X_k) \rangle.
\]  
Using $ M_k = X_k - \theta h b(X_k) $, we have  
\[
\begin{aligned}
	2h \mathbb{E}\langle M_k, b(X_k) \rangle &= 2h \mathbb{E}\langle X_k - \theta h b(X_k), b(X_k) \rangle \\
	&= 2h \mathbb{E}\langle X_k, b(X_k) \rangle - 2\theta h^2 \mathbb{E}|b(X_k)|^2.
\end{aligned}
\]  
Thus,  
\[
\begin{aligned}
	\mathbb{E}|M_{k+1}|^2 &= \mathbb{E}|M_k|^2 + h^2 \mathbb{E}|b(X_k)|^2 + h \mathbb{E}|g(X_k)|^2 + 2h \mathbb{E}\langle X_k, b(X_k) \rangle - 2\theta h^2 \mathbb{E}|b(X_k)|^2 \\
	&= \mathbb{E}|M_k|^2 + 2h \mathbb{E}\langle X_k, b(X_k) \rangle + h \mathbb{E}|g(X_k)|^2 - (2\theta - 1)h^2 \mathbb{E}|b(X_k)|^2.
\end{aligned}
\]  
By Assumption \ref{assum:2.2}, we have  
\[
\begin{aligned}
	2\langle X_k, b(X_k) \rangle + |g(X_k)|^2 &= 2\langle X_k, b(X_k) \rangle + l_2 |g(X_k)|^2 - (l_2 - 1)|g(X_k)|^2 \\
	&\leq c_1 - \alpha |X_k|^2 - (l_2 - 1)|g(X_k)|^2.
\end{aligned}
\]  
Therefore,  
\[
\begin{aligned}
	\mathbb{E}|M_{k+1}|^2 &\leq \mathbb{E}|M_k|^2 + h\left(c_1 - \alpha |X_k|^2 - (l_2 - 1)|g(X_k)|^2\right) - (2\theta - 1)h^2 \mathbb{E}|b(X_k)|^2 \\
	&= \mathbb{E}|M_k|^2 + c_1 h - \alpha h \mathbb{E}|X_k|^2 - h(l_2 - 1)\mathbb{E}|g(X_k)|^2 - (2\theta - 1)h^2 \mathbb{E}|b(X_k)|^2.
\end{aligned}
\]  
Since $ l_2 \geq 1 $, we have $ -h(l_2 - 1)\mathbb{E}|g(X_k)|^2 \leq 0 $, so  
\[
\mathbb{E}|M_{k+1}|^2 \leq \mathbb{E}|M_k|^2 + c_1 h - \alpha h \mathbb{E}|X_k|^2 - (2\theta - 1)h^2 \mathbb{E}|b(X_k)|^2.
\]  
Let $\theta > \frac{1}{2}$ and define  
\[
\rho_0 = \left\{\dfrac{\alpha}{2} \wedge \dfrac{2\theta - 1}{2\theta^2}\right\}.
\]  
Since $\theta > \frac{1}{2}$, we have $\frac{2\theta - 1}{2\theta^2} > 0$, and $\alpha > 0$ by Assumption 2.2, so $\rho_0 > 0$. This choice of $\rho_0$ ensures the inequalities $2\rho_0 \leq \alpha$ and $2\rho_0\theta^2 \leq 2\theta - 1$ hold.

We assume the stepsize $h$ satisfies  
\[
0 < \rho_0 h < 1,
\]  
which is feasible because $\rho_0 > 0$ and $h$ can be chosen sufficiently small. This guarantees $0 < 1 - \rho_0 h < 1$, a critical condition for convergence of the geometric series in the stability analysis.

Now, from $ M_k = X_k - \theta h b(X_k) $, we derive  
\[
|M_k|^2 = |X_k - \theta h b(X_k)|^2 \leq 2|X_k|^2 + 2\theta^2 h^2 |b(X_k)|^2 
\]  
Multiplying both sides by $ \rho_0 h $, we get  
\[
\rho_0 h |M_k|^2 \leq 2\rho_0 h |X_k|^2 + 2\rho_0 \theta^2 h^3 |b(X_k)|^2.
\]  
By the definition of $ \rho_0 $ and $h\in(0,1)$:  
\begin{itemize}
	\item $ 2\rho_0 \leq \alpha \implies 2\rho_0 h |X_k|^2 \leq \alpha h |X_k|^2 $,  
	\item $ 2\rho_0 \theta^2 h^2 \leq (2\theta - 1)h^2 \implies 2\rho_0 \theta^2 h^3 |b(X_k)|^2 \leq (2\theta - 1)h^2 |b(X_k)|^2 $,  
\end{itemize}
so  
\[
\rho_0 h |M_k|^2 \leq \alpha h |X_k|^2 + (2\theta - 1)h^2 |b(X_k)|^2.
\]  
Substituting back into the expression for $ \mathbb{E}|M_{k+1}|^2 $,  
\[
\begin{aligned}
	\mathbb{E}|M_{k+1}|^2 &\leq \mathbb{E}|M_k|^2 + c_1 h - \alpha h |X_k|^2 - (2\theta - 1)h^2 |b(X_k)|^2 \\
	&= (1-\rho_0 h)\mathbb{E}|M_k|^2 + \rho_0 h\mathbb{E}|M_k|^2 + c_1 h - \alpha h |X_k|^2 - (2\theta - 1)h^2 |b(X_k)|^2\ \\
	&\leq (1 - \rho_0 h) \mathbb{E}|M_k|^2 + c_1 h.
\end{aligned}
\]  
Iterating this inequality,  
\[
\begin{aligned}
	\mathbb{E}|M_{k+1}|^2 &\leq (1 - \rho_0 h) \mathbb{E}|M_k|^2 + c_1 h \\
	&\leq (1 - \rho_0 h)^2 \mathbb{E}|M_{k-1}|^2 + c_1 h (1 - \rho_0 h) + c_1 h \\
	&\leq \cdots \\
	&\leq (1 - \rho_0 h)^{k+1} \mathbb{E}|M_0|^2 + c_1 h \sum_{j=0}^k (1 - \rho_0 h)^j.
\end{aligned}
\]  
Since $ 0 < 1 - \rho_0 h < 1 $ (by the stepsize condition $ 0 < \rho_0 h < 1 $), the geometric series converges:  
\[
\sum_{j=0}^k (1 - \rho_0 h)^j = \frac{1 - (1 - \rho_0 h)^{k+1}}{\rho_0 h} \leq \frac{1}{\rho_0 h}.
\]  
Thus, for any $k \geq 0$, we have  
\[
\mathbb{E}|M_k|^2 \leq (1 - \rho_0 h)^k \mathbb{E}|M_0|^2 + \frac{c_1}{\rho_0}.
\]  
Noting $ M_0 = X_0 - \theta h b(X_0) $, we obtain  
\[
\mathbb{E}|M_k|^2 \leq (1 - \rho_0 h)^k \mathbb{E}|X_0 - \theta h b(X_0)|^2 + \frac{c_1}{\rho_0}.
\]  
From the definition of $ M_k $,  
\[
|M_k|^2 = |X_k - \theta h b(X_k)|^2 = |X_k|^2 - 2\theta h \langle X_k, b(X_k) \rangle + \theta^2 h^2 |b(X_k)|^2.
\]  
Using Assumption \ref{assum:2.2} again, $ 2\langle X_k, b(X_k) \rangle \leq c_1 - \alpha |X_k|^2 - l_2 |g(X_k)|^2 \leq c_1 - \alpha |X_k|^2 $, so  
\[
\begin{aligned}
	|M_k|^2 &\geq |X_k|^2 - \theta h (c_1 - \alpha |X_k|^2) + \theta^2 h^2 |b(X_k)|^2 \\
	&\geq (1 + \theta \alpha h)|X_k|^2 - c_1 \theta h.
\end{aligned}
\]  
Taking expectations and rearranging,  
\[
(1 + \theta \alpha h) \mathbb{E}|X_k|^2 - c_1 \theta h \leq \mathbb{E}|M_k|^2 \leq (1 - \rho_0 h)^k \mathbb{E}|X_0 - \theta h b(X_0)|^2 + \frac{c_1}{\rho_0}.
\]  
Therefore,  
\[
\mathbb{E}|X_k|^2 \leq \frac{(1 - \rho_0 h)^k}{1 + \theta \alpha h} \mathbb{E}|X_0 - \theta h b(X_0)|^2 + \frac{c_1}{\rho_0 (1 + \theta \alpha h)} + \frac{c_1 \theta h}{1 + \theta \alpha h}.
\]  
So under Assumption \ref{assum:2.2} and stepsize condition $ 0 < \rho_0 h < 1 $, then the solution obeys
\[
\mathbb{E}|X_k|^2 \leq C_1,
\]
where $C_1 > 0$ is a constant .
\end{proof}
\begin{lemma}\label{lem:3.2}
Assume that Assumptions \ref{assum:2.1} and \ref{assum:2.3} hold. Let $\{X_k\}$ and $\{Y_k\}$ be two numerical solutions of the stochastic theta method \eqref{eq:NumSDE} with different initial values $x_0$ and $y_0$. For any $h\in(0,1)$ satisfying $0 < h < 1/\tilde{\rho}_0$, then
\[
\mathbb{E}|X_k-Y_k|^2 \leq e^{-\tilde{\rho}_0 h k}(1+C_2) |X_0 - Y_0|^2,
\]
where $\tilde{\rho}_0 = \left\{\dfrac{\beta}{2} \wedge \dfrac{2\theta-1}{2\theta^2}\right\}$ and $C_2 > 0$ is a constant.
\end{lemma}

\begin{proof}
Define $D_k = X_k - Y_k$ and $R_k = D_k - \theta h[b(X_k) - b(Y_k)]$. The sequences $\{X_k\}$ and $\{Y_k\}$ satisfy
\begin{align*}
	X_{k+1} &= X_k + \theta h\, b(X_{k+1}) + (1-\theta)h\, b(X_k) + g(X_k)\,\Delta W_k, \\
	Y_{k+1} &= Y_k + \theta h\, b(Y_{k+1}) + (1-\theta)h\, b(Y_k) + g(Y_k)\,\Delta W_k,
\end{align*}
where $\Delta W_k \sim \mathcal{N}(0,h)$ is the Brownian increment. Subtracting the two equations gives
\[
D_{k+1} = D_k + \theta h\big[b(X_{k+1})-b(Y_{k+1})\big] + (1-\theta)h\big[b(X_k)-b(Y_k)\big] + \big[g(X_k)-g(Y_k)\big]\Delta W_k.
\]
Hence, by the definition of $R_{k+1}$, we have
\[
R_{k+1} = D_{k+1} - \theta h\big[b(X_{k+1})-b(Y_{k+1})\big] = D_k + (1-\theta)h\big[b(X_k)-b(Y_k)\big] + \big[g(X_k)-g(Y_k)\big]\Delta W_k,
\]
which can be rewritten as
\[
R_{k+1} = R_k + h\big[b(X_k)-b(Y_k)\big] + \big[g(X_k)-g(Y_k)\big]\Delta W_k.
\]
Squaring both sides yields and taking expectation , we obtain
\[
\mathbb{E}|R_{k+1}|^2 = \mathbb{E}|R_k|^2 + 2h\,\mathbb{E}\langle R_k,\, b(X_k)-b(Y_k)\rangle + h^2\,\mathbb{E}|b(X_k)-b(Y_k)|^2 + h\,\mathbb{E}|g(X_k)-g(Y_k)|^2.
\]
Substituting $R_k = D_k - \theta h[b(X_k)-b(Y_k)]$ into the inner product term yields:
\[
\langle R_k,\, b(X_k)-b(Y_k)\rangle = \langle D_k,\, b(X_k)-b(Y_k)\rangle - \theta h\,|b(X_k)-b(Y_k)|^2.
\]
By Assumption \ref{assum:2.3},
\begin{align*}
	2\langle D_k, b(X_k) - b(Y_k) \rangle + l_1 |g(X_k) - g(Y_k)|^2 \leq -\beta |D_k|^2,
\end{align*}
which implies
\begin{align*}
	2\langle D_k, b(X_k) - b(Y_k) \rangle \leq -\beta |D_k|^2 - l_1 |g(X_k) - g(Y_k)|^2.
\end{align*}
Substitute this inequality:
\begin{align*}
	\mathbb{E}|R_{k+1}|^2 &\leq \mathbb{E}|R_k|^2 + h\mathbb{E}|g(X_k) - g(Y_k)|^2 + h\left(-\beta\mathbb{E}|D_k|^2 - l_1\mathbb{E}|g(X_k) - g(Y_k)|^2\right) \\
	&\quad + (1-2\theta)h^2\mathbb{E}|b(X_k) - b(Y_k)|^2 \\
	&= \mathbb{E}|R_k|^2 - \beta h\mathbb{E}|D_k|^2 - (l_1-1)h\mathbb{E}|g(X_k) - g(Y_k)|^2 \\
	&\quad + (1-2\theta)h^2\mathbb{E}|b(X_k) - b(Y_k)|^2.
\end{align*}
Since $l_1 \geq 1$ , we get $-(l_1-1)h\mathbb{E}|g(X_k) - g(Y_k)|^2 \leq 0$. 
Therefore,
\[
\mathbb{E}|R_{k+1}|^2 \leq \mathbb{E}|R_k|^2 - \beta h\mathbb{E}|D_k|^2 + (1-2\theta)h^2\mathbb{E}|b(X_k) - b(Y_k)|^2.
\]
Now define
\[
\tilde{\rho}_0 = \left\{\frac{\beta}{2} \wedge \frac{2\theta-1}{2\theta^2}\right\}.
\]
Since $\theta \in (\frac{1}{2},1]$ and $\beta > 0$, we have $\tilde{\rho}_0 > 0$ and $0 < \tilde{\rho}_0 h < 1$ for sufficiently small $h$. From the definition of $\tilde{\rho}_0$, we have $2\tilde{\rho}_0 \leq \beta$ and $2\tilde{\rho}_0 \theta^2 \leq 2\theta - 1$. Using the inequality $|R_k|^2 \leq 2|D_k|^2 + 2\theta^2 h^2 |b(X_k)-b(Y_k)|^2$, we get
\[
\tilde{\rho}_0 h\,|R_k|^2 \leq 2\tilde{\rho}_0 h\,|D_k|^2 + 2\tilde{\rho}_0 \theta^2 h^3\,|b(X_k)-b(Y_k)|^2 \leq \beta h\,|D_k|^2 + (2\theta-1)h^2\,|b(X_k)-b(Y_k)|^2.
\]
Rearranging this yields
\[
\beta h\,|D_k|^2 \geq \tilde{\rho}_0 h\,|R_k|^2 - (2\theta-1)h^2\,|b(X_k)-b(Y_k)|^2.
\]
Substituting this lower bound into the inequality for $\mathbb{E}|R_{k+1}|^2$:
\begin{align*}
	\mathbb{E}|R_{k+1}|^2 &\leq \mathbb{E}|R_k|^2 - \tilde{\rho}_0 h\,\mathbb{E}|R_k|^2 + (2\theta-1)h^2\,\mathbb{E}|b(X_k)-b(Y_k)|^2 + (1-2\theta)h^2\,\mathbb{E}|b(X_k)-b(Y_k)|^2 \\
	&= (1-\tilde{\rho}_0 h)\,\mathbb{E}|R_k|^2.
\end{align*}
Iterating this inequality from step $0$ to $k$:
\[
\mathbb{E}|R_k|^2 \leq (1-\tilde{\rho}_0 h)^k\,\mathbb{E}|R_0|^2.
\]
Using the elementary inequality $(1-x)^k \leq e^{-xk}$ for $x \in (0,1)$, we obtain
\[
\mathbb{E}|R_k|^2 \leq e^{-\tilde{\rho}_0 h k}\,\mathbb{E}|R_0|^2.
\]
Noting $R_0 = D_0 - \theta h[b(x_0) - b(y_0)] = (x_0 - y_0) - \theta h[b(x_0) - b(y_0)]$, we obtain
\[
\mathbb{E}|R_k|^2 \leq e^{-\tilde{\rho}_0 h k} \mathbb{E}\left|(x_0 - y_0) - \theta h[b(x_0) - b(y_0)]\right|^2.
\]
Retracing $D_k$, by Assumption \ref{assum:2.3}, we have
\begin{align*}
	|R_k|^2 &= |D_k - \theta h[b(X_k) - b(Y_k)]|^2 \\
	&= |D_k|^2 - 2\theta h\langle D_k, b(X_k) - b(Y_k) \rangle + \theta^2 h^2 |b(X_k) - b(Y_k)|^2 \\
	&\geq |D_k|^2 + \beta\theta h|D_k|^2 \\
	&= (1 + \beta\theta h)|D_k|^2.
\end{align*}
Therefore,
\begin{align*}
	(1 + \beta \theta h)\mathbb{E} |D_k|^2 &\leq \mathbb{E}|R_k|^2 \leq e^{-\tilde{\rho}_0 h k} \mathbb{E}\left[ (x_0 - y_0) - \theta h (b(x_0) - b(y_0)) \right]^2 \\
	\implies \mathbb{E}|D_k|^2 &\leq \frac{e^{-\tilde{\rho}_0 h k}}{(1 + \beta \theta h)} \mathbb{E}\left[ (x_0 - y_0) - \theta h (b(x_0) - b(y_0)) \right]^2.
\end{align*}
Noting $D_k = X_k - Y_k$, we obtain 
\begin{align*}
	\mathbb{E}|X_k - Y_k|^2 = e^{-\tilde{\rho}_0 h k} \mathbb{E}\left[ (x_0 - y_0) - \theta h (b(x_0) - b(y_0)) \right]^2 
	\leq e^{-\tilde{\rho}_0 h k}(1 + C_2) |x_0 - y_0|^2.
\end{align*}
So under Assumption \ref{assum:2.3} and stepsize condition $ 0 < \tilde{\rho}_0 h < 1 $, then the solution obeys
\begin{align*}
	\mathbb{E}|X_k - Y_k|^2 \leq e^{-\tilde{\rho}_0 h k} (1 + C_2) |x_0 - y_0|^2,
\end{align*}
where $\tilde{\rho}_0 = \left\{\dfrac{\beta}{2}\wedge \dfrac{2\theta-1}{2\theta^2}\right\}$ and $C_2 > 0$ is a constant.
\end{proof}

Now, we are ready to present the main result on the existence and uniqueness of the numerical invariant measure.
\begin{theorem}\label{thm:1}
Let Assumptions \ref{assum:2.1} , \ref{assum:2.3}and \ref{assum:2.2} hold, for any $h \in (0, 1)$ satisfying
\[
h < h^* := \frac{1}{\rho_0} \wedge \frac{1}{\tilde{\rho}_0},
\]
where $\tilde{\rho}_0 = \left\{\dfrac{\beta}{2} \wedge \dfrac{2\theta-1}{2\theta^2}\right\}$ and 
$\rho_0 = \left\{ \dfrac{\alpha}{2}\wedge \dfrac{2\theta - 1}{2\theta^2} \right\}$,
the stochastic theta method \eqref{eq:NumSDE} converges in the $p$-Wasserstein distance to a unique invariant measure $\widehat{\pi} \in \mathcal{P}(\mathbb{R}^d)$ with some exponential rate $\xi_1 > 0$ for any $p \in (0, 1)$.
\end{theorem}

\begin{proof}
Due to the Chebyshev inequality and Lemma \ref{lem:3.1}, for any initial value $x_0 \in \mathbb{R}^n$, we obtain that the sequence $\{\delta_{x_0} \hat{\mathbb{P}}_k^h\}_{k \geq 1}$ is tight, where $\delta_{x_0} \hat{\mathbb{P}}_k^h = \hat{\mathbb{P}}_k^h(x_0, \cdot)$ denotes the probability measure of $X_k$ starting from $x_0$. Then, a subsequence that converges weakly to an invariant measure $\hat{\pi} \in \mathcal{P}(\mathbb{R}^n)$ can be extracted. By Lemma \ref{lem:3.2}, under the stepsize condition $0 < h < h^*$, we have
\[
\mathbb{E}|X_k - Y_k|^2 \leq (1 + C_2)|x_0 - y_0|^2 e^{-\tilde{\rho}_0 h k}.
\]
By the H\"older inequality and the definition of the $p$-Wasserstein distance, for any $p \in (0, 1)$,
\begin{equation}\label{eq:contraction}
	W_p(\delta_{x_0} \hat{\mathbb{P}}_k^h, \delta_{y_0} \hat{\mathbb{P}}_k^h)
	\leq \mathbb{E}|X_k - Y_k|^p
	\leq \left(\mathbb{E}|X_k - Y_k|^2\right)^{p/2}
	\leq (1 + C_2)^{p/2} |x_0 - y_0|^p e^{-p \tilde{\xi}_1 k h},
\end{equation}
where $\tilde{\xi}_1 = \tilde{\rho}_0 / 2$.

Thanks to Lemma \ref{lem:3.1} and the Kolmogorov-Chapman equation, for any $k, l > 0$ and $p \in (0, 1)$, we have
\[
\begin{aligned}
	W_p(\delta_{x_0} \hat{\mathbb{P}}_k^h, \delta_{x_0} \hat{\mathbb{P}}_{(k+l)}^h)
	&= W_p(\delta_{x_0} \hat{\mathbb{P}}_k^h, \delta_{x_0} \hat{\mathbb{P}}_k^h \hat{\mathbb{P}}_l^h) \\
	&\leq \int_{\mathbb{R}^n} W_p(\delta_{x_0} \hat{\mathbb{P}}_k^h, \delta_{y_0} \hat{\mathbb{P}}_k^h) \hat{\mathbb{P}}_l^h(x_0, dy_0) \\
	&\leq \int_{\mathbb{R}^n} (1 + C_2)^{p/2} |x_0 - y_0|^p e^{-p \tilde{\xi}_1 k h} \hat{\mathbb{P}}_l^h(x_0, dy_0) \\
	&\leq 2^{p/2} (1 + C_2)^{p/2} \left(|x_0|^p + \mathbb{E}|X_l|^p\right) e^{-p \tilde{\xi}_1 k h} \\
	&\leq 2^{p/2} (1 + C_2)^{p/2} \left(|x_0|^p + (\mathbb{E}|X_l|^2)^{p/2}\right) e^{-p \tilde{\xi}_1 k h}.
\end{aligned}
\]
By Lemma \ref{lem:3.1}, $\mathbb{E}|X_l|^2 \leq C_1$ for all $l \geq 0$. Therefore,
\begin{equation}\label{eq:cauchy}
	W_p(\delta_{x_0} \hat{\mathbb{P}}_k^h, \delta_{x_0} \hat{\mathbb{P}}_{(k+l)}^h) \leq K(p) e^{-p \tilde{\xi}_1 k h},
\end{equation}
where $K(p) := 2^{p/2}(1 + C_2)^{p/2}\left(|x_0|^p + C_1^{p/2}\right)$. Now, letting $l \to \infty$ in \eqref{eq:cauchy}, since $\delta_{x_0} \hat{\mathbb{P}}_{(k+l)}^h$ converges weakly to $\hat{\pi}$, we have
\begin{equation}\label{eq:convergence}
	W_p(\delta_{x_0} \hat{\mathbb{P}}_k^h, \hat{\pi}) \leq K(p) e^{-p \tilde{\xi}_1 k h}.
\end{equation}
Moreover, we have $W_p(\delta_{x_0} \hat{\mathbb{P}}_k^h, \hat{\pi}) \to 0$ as $k \to \infty$, which guarantees that $\hat{\pi}$ is an invariant measure of $\{\delta_{x_0} \hat{\mathbb{P}}_k^h\}$.

Now, assume that $\hat{\pi}_1 \in \mathcal{P}(\mathbb{R}^n)$ is an invariant measure of $X_k$ with initial value $x_0$ and $\hat{\pi}_2 \in \mathcal{P}(\mathbb{R}^n)$ is an invariant measure of $X_k$ with initial value $y_0$. We can see that
\[
W_p(\hat{\pi}_1, \hat{\pi}_2) \leq \int_{\mathbb{R}^n \times \mathbb{R}^n} W_p(\delta_{x_0} \hat{\mathbb{P}}_k^h, \delta_{y_0} \hat{\mathbb{P}}_k^h) \nu(dx_0, dy_0),
\]
for any $x_0,y_0\in\mathbb{R}^n$ with $x_0\neq y_0$, and $\nu$ is a  coupling of $\hat{\pi}_1$ and $\hat{\pi}_2$. Consequently, it follows from \eqref{eq:contraction} that the stochastic theta method admits a unique invariant measure.
\end{proof}

Motivated by the argument in \cite{BaoShaoYuan2016}, we derive the following theorem concerning the invariant measure of the underlying equation. Since the proof follows the same line of reasoning as that in \cite{BaoShaoYuan2016}, we omit the proof.

\begin{theorem}\label{thm:2}
Suppose that Assumptions \ref{assum:2.1}--\ref{assum:2.2} hold. Then the solution of \eqref{eq:SDE} converges in the $p$-Wasserstein distance to a unique invariant measure $\pi$ with some exponential rate $\xi_2>0$.
\end{theorem}

We now state the convergence result of the numerical invariant measure to its underlying counterpart.

\begin{theorem}\label{thm:3}
Suppose that Assumptions 2.1, 2.2 and 2.3 hold.  Then the numerical invariant measure $\hat{\pi}$ converges to the underlying invariant measure $\pi$ in the $p$-Wasserstein distance, that is, for any $h \in (0, h^*)$, we have
\[
W_p(\hat{\pi}, \pi) \leq Ch^{p/2}.
\]
\end{theorem}
\begin{proof}
Thanks to the triangle inequality, we have
\[
W_p(\hat{\pi}, \pi) 
\leq W_p\bigl(\hat{\pi}, \delta_{x_0} \hat{\mathbb{P}}_k^h)
+ W_p\bigl(\delta_{x_0} \mathbb{P}_{k}, \pi\bigr) 
+ W_p\bigl(\delta_{x_0} \mathbb{P}_{k}, \delta_{x_0} \hat{\mathbb{P}}_{k}^h\bigr).
\]

It follows from Theorems \ref{thm:1} and \ref{thm:2} that $\hat{\mathbb{P}}_{k}^h$ and $\mathbb{P}_{k}^h$ converge to $\hat{\pi}$ and $\pi$, respectively. Therefore,
\[
W_p(\hat{\pi}, \delta_{x_0} \hat{\mathbb{P}}_k^h) \leq C e^{-p \tilde{\xi}_1 k h} \quad \text{and} \quad
W_p\bigl(\delta_{x_0} \mathbb{P}_{k}, \pi\bigr) 
\leq C e^{-r\xi_2 kh},
\]
where $C$ is some constant. 
For the last term, by the strong convergence result of the stochastic theta method in \cite{WangWuDong2020}, we have
\begin{align*}
	W_p(\delta_{x_0}\mathbb{P}_k,\delta_{x_0}\hat{\mathbb{P}}_k^h)
	&\leq \mathbb{E}|X(t_k)-{X}_k|^p  \\
	&\leq \left(\mathbb{E}|X(t_k)-{X}_k|^2\right)^{p/2} \\
	&\leq Ch^{p/2}.
\end{align*}
Letting $k\to\infty$, we obtain
\begin{equation*}
	W_p(\hat{\pi},\pi)\leq Ch^{p/2}.
\end{equation*}
This completes the proof.
\end{proof}

\section{Numerical examples}\label{sec:numexp}
We present two numerical examples in this section to demonstrate the theoretical
results.
\begin{example}\label{Ex1}
Consider a scalar mean-reverting type model with super-linear coefficients
\begin{equation*}
	dX(t)=(1-2X(t)-5X^3(t))dt+X^2(t)dW(t), \quad X(0)=5.
\end{equation*}
\end{example}
We have $b(x)=1-2x-5x^3$ and $g(x)=x^2$. It is easy to verify that Assumptions \ref{assum:2.1}-\ref{assum:2.2} hold. Indeed, Assumption \ref{assum:2.1} holds with $q=3$ and $L=8$. Assumption \ref{assum:2.3} holds with $l_1=4$ and $\beta=4$. Moreover, Assumption \ref{assum:2.2} holds with $l_2=9$, $c_1=1$ and $\alpha=3$. Therefore, according
to our theorems in Section \ref{sec:results} there exists a unique invariant measure for the stochastic theta method.

One thousand sample paths are simulated with $X_0 = 5$, $h = 0.01$ and $T=100$, which are then used to construct empirical density functions at different time points. It is clear to see from the left plot in Figure \ref{fig:Example1} that the shapes of empirical density functions at $t = 0.1$, $t = 0.3$ and $t = 0.5$ are quite
different but the ones at $t = 4$ and $t = 10$ are much more similar, which indicates the existence of the invariant measure. The right plot of Figure \ref{fig:Example1} shows that, at $t=50$, the empirical density functions generated from different initial values  $-5$, $5$ and $15$ are nearly indistinguishable, which provides numerical evidence for the uniqueness of the invariant measure.

\begin{figure}[H]
\centering
\begin{subfigure}{0.49\textwidth}
	\centering
	\includegraphics[width=\textwidth]{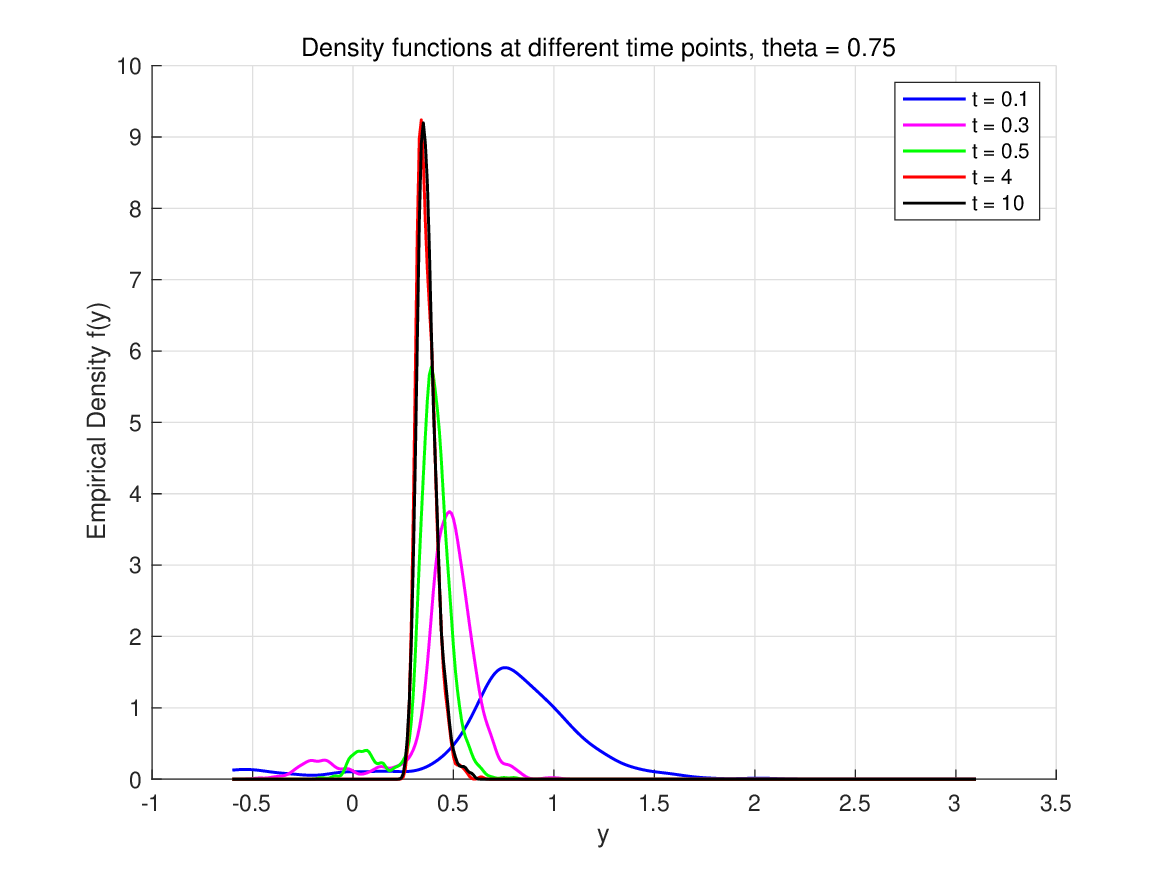}
\end{subfigure}
\hfill
\begin{subfigure}{0.49\textwidth}
	\centering
	\includegraphics[width=\textwidth]{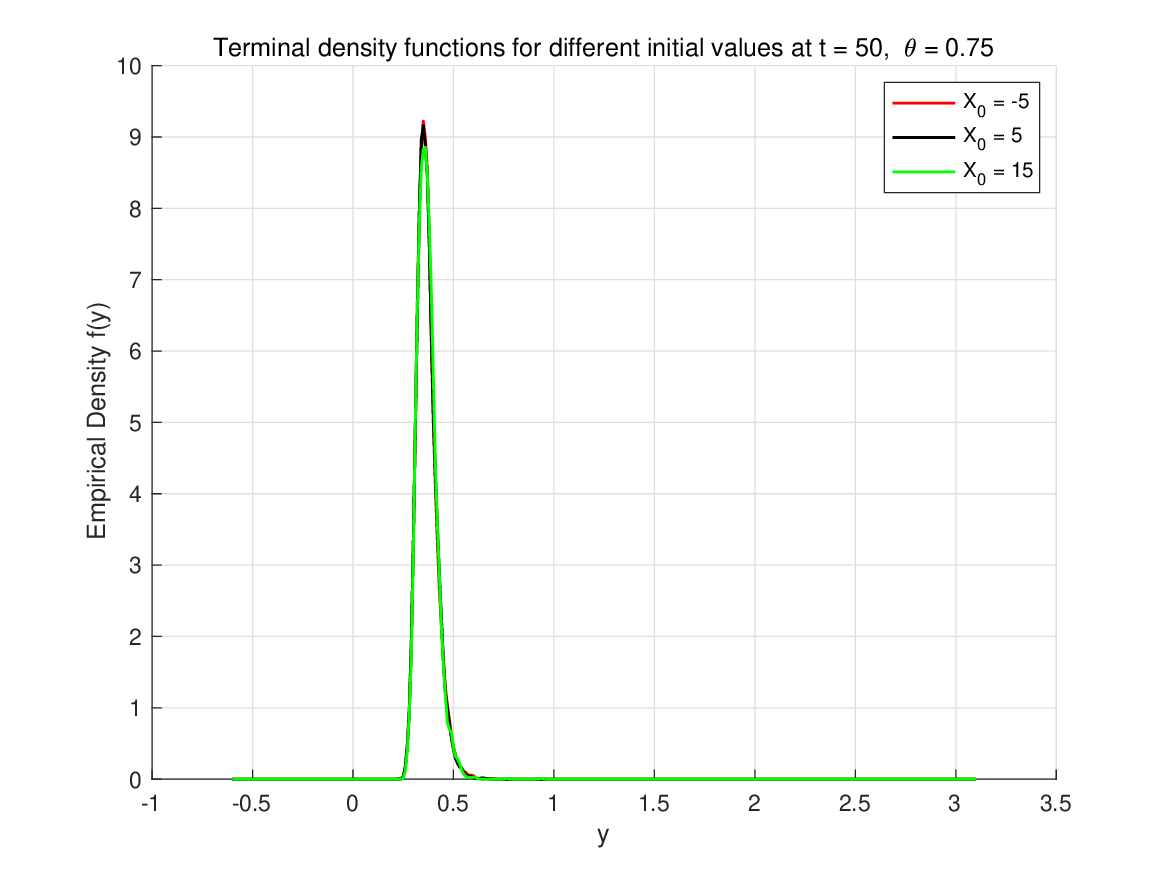}
\end{subfigure}
\caption{Left: Empirical density functions at different time points; Right: Empirical density functions at $t=50$ with different initial values}
\label{fig:Example1}
\end{figure}

\begin{figure}[H]
\centering
\includegraphics[width=0.7\linewidth]{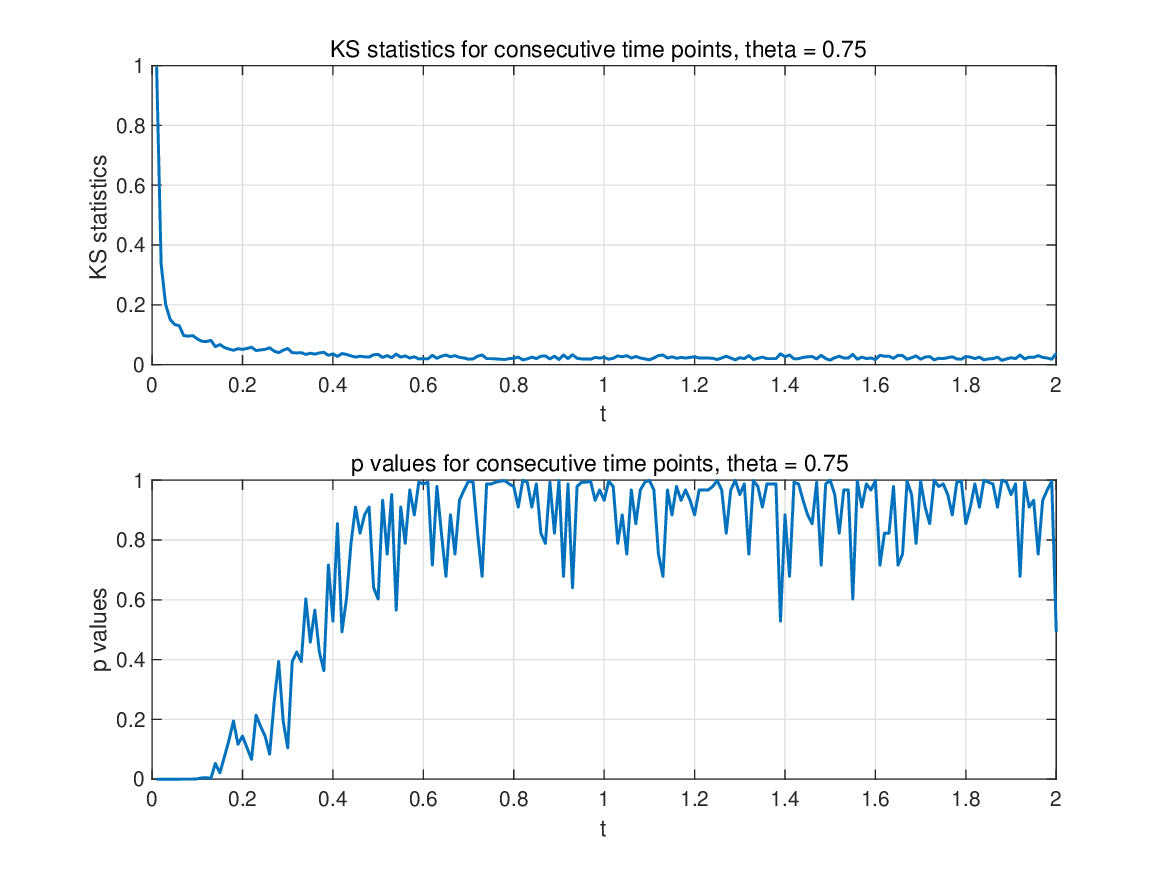}
\caption{K-S tests for samples at consecutive time points}
\label{fig:Example1-2}
\end{figure}
Figure \ref{fig:Example1-2} presents the K-S statistics and the corresponding $p$ values for the empirical distributions of the numerical solution at consecutive time points. The decay of the K-S statistics together with the increase of the $p$ values indicates that the difference between consecutive empirical distributions gradually vanishes as time evolves, which suggests the existence of an invariant measure for the numerical solution.

\begin{example}
Consider the following two dimensional SDE
\begin{align*}
	\left\{\begin{array}{ll}
		dX_1(t) = (1 - 2X_1(t) + 0.5X_2(t) - X^3_1(t))dt + 0.3X^2_1(t) dW_1(t), \\
		dX_2(t) = (1 + 0.5X_1(t) - 2X_2(t) - X^3_2(t))dt + 0.3X^2_2(t) dW_2(t),
	\end{array}\right.
\end{align*}
where $X_1(0)=X_2(0)=5$.
\end{example}
We now verify that Assumptions \ref{assum:2.1}--\ref{assum:2.2} are satisfied for this example.

Set $x=(x_1,x_2)^{\mathrm{T}}$, we have 
\begin{equation*}
b(x) =
\begin{pmatrix}
	1 - 2x_1 + 0.5x_2 - x_1^3 \\
	1 + 0.5x_1 - 2x_2 - x_2^3
\end{pmatrix},
\qquad
g(x) =
\begin{pmatrix}
	0.3x_1^2 & 0 \\
	0 & 0.3x_2^2
\end{pmatrix}.
\end{equation*}

First, we verify Assumption \ref{assum:2.1}. For any $x,y\in\mathbb{R}^2$, we have
\begin{align*}
b_1(x)-b_1(y)=-2(x_1-y_1)+0.5(x_2-y_2)-(x_1^3-y_1^3),\\
b_2(x)-b_2(y)=0.5(x_1-y_1)-2(x_2-y_2)-(x_2^3-y_2^3).
\end{align*}
Moreover,
\begin{equation*}
g(x)-g(y) =
\begin{pmatrix}
	0.3(x_1^2-y_1^2) & 0 \\
	0 & 0.3(x_2^2-y_2^2)
\end{pmatrix},
\end{equation*}
Since $x_i^3-y_i^3=(x_i-y_i)(x_i^2+x_iy_i+y_i^2)$ and $x_i^2-y_i^2=(x_i-y_i)(x_i+y_i)$, there exists a constant $C>0$ such that
\begin{equation*}
|b(x)-b(y)|\vee |g(x)-g(y)|\leq C(1+|x|^2+|y|^2)|x-y|.
\end{equation*}
Therefore, Assumption \ref{assum:2.1} holds with $q=3$.

Next, we verify Assumption \ref{assum:2.3}. Let
$\delta=x-y=(\delta_1,\delta_2)^T$, where $\delta_i=x_i-y_i$ for $i=1,2$.
By direct calculation, we have
\begin{align*}
2\langle x-y,b(x)-b(y)\rangle
=&-4\delta_1^2-4\delta_2^2+2\delta_1\delta_2  \\
&-2\delta_1^2(x_1^2+x_1y_1+y_1^2)
-2\delta_2^2(x_2^2+x_2y_2+y_2^2).
\end{align*}
Moreover,
\begin{equation*}
|g(x)-g(y)|^2
=
0.09\delta_1^2(x_1+y_1)^2
+
0.09\delta_2^2(x_2+y_2)^2.
\end{equation*}
Taking $l_1=4$, we obtain
\begin{align*}
2\langle x-y,b(x)-b(y)\rangle
+4|g(x)-g(y)|^2 =&-4\delta_1^2-4\delta_2^2+2\delta_1\delta_2  \\
&+\delta_1^2\left[-2(x_1^2+x_1y_1+y_1^2)
+0.36(x_1+y_1)^2\right]  \\
&+\delta_2^2\left[-2(x_2^2+x_2y_2+y_2^2)
+0.36(x_2+y_2)^2\right].
\end{align*}
For any $a,b\in\mathbb{R}$,
\begin{equation*}
-2(a^2+ab+b^2)+0.36(a+b)^2
=
-1.64a^2-1.28ab-1.64b^2
\leq 0.
\end{equation*}
Therefore,
\begin{equation*}
2\langle x-y,b(x)-b(y)\rangle
+4|g(x)-g(y)|^2
\leq
-4\delta_1^2-4\delta_2^2+2\delta_1\delta_2.
\end{equation*}
Using $2\delta_1\delta_2\leq \delta_1^2+\delta_2^2$, we further get
\begin{equation*}
2\langle x-y,b(x)-b(y)\rangle
+4|g(x)-g(y)|^2
\leq
-3(\delta_1^2+\delta_2^2)
=
-3|x-y|^2.
\end{equation*}
Thus, Assumption \ref{assum:2.3} holds with $l_1=4$ and $\beta=3$.

Finally, we verify Assumption \ref{assum:2.2}. For any $x\in\mathbb{R}^2$, we have
\begin{align*}
2\langle x,b(x)\rangle
=&2x_1(1-2x_1+0.5x_2-x_1^3)
+2x_2(1+0.5x_1-2x_2-x_2^3)\\
=&2x_1+2x_2-4x_1^2-4x_2^2+2x_1x_2
-2x_1^4-2x_2^4.
\end{align*}
Moreover,
\begin{equation*}
|g(x)|^2=0.09(x_1^4+x_2^4).
\end{equation*}
Taking $l_2=9$, we have
\begin{align*}
2\langle x,b(x)\rangle+9|g(x)|^2
=&2x_1+2x_2-4x_1^2-4x_2^2+2x_1x_2\\
&-1.19(x_1^4+x_2^4).
\end{align*}
Since
\begin{equation*}
2x_1+2x_2\leq x_1^2+x_2^2+2
\end{equation*}
and
\begin{equation*}
2x_1x_2\leq x_1^2+x_2^2,
\end{equation*}
it follows that
\begin{align*}
2\langle x,b(x)\rangle+9|g(x)|^2
&\leq 2-2x_1^2-2x_2^2-1.19(x_1^4+x_2^4)\\
&\leq 2-2|x|^2.
\end{align*}
Thus, Assumption \ref{assum:2.2} holds with $l_2=9$, $c_1=2$ and $\alpha=2$.

Figure \ref{fig:Example2} presents the empirical density functions of the numerical solution at $t=0.1$, $t=0.3$, $t=4$ and $t=10$. It can be seen that the shape of the empirical density changes significantly at the early stage. As time evolves, the mass of the numerical solution becomes increasingly concentrated in a bounded region, and the empirical density functions at later time points become much more stable. 

Figure \ref{fig:Example2-2} displays the terminal samples of the numerical solution at $T=10$ with $\theta=0.75$. The sample points are concentrated around a small region near the equilibrium area, and no divergent sample paths are observed. This further illustrates the long-time stability of the proposed stochastic theta method and supports the convergence of the numerical distribution to an invariant measure.

\begin{figure}[H]
\centering
\begin{subfigure}{0.49\textwidth}
	\centering
	\includegraphics[width=\textwidth]{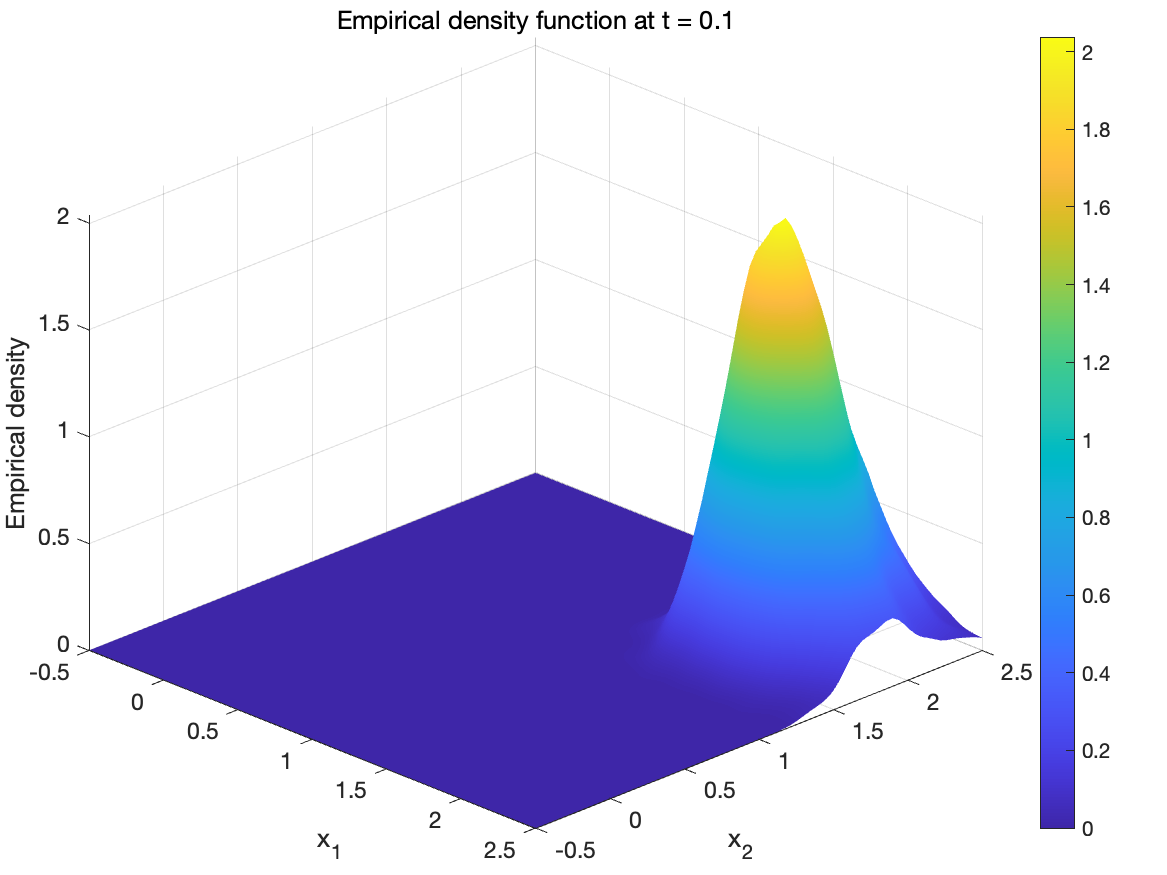}
\end{subfigure}
\hfill
\begin{subfigure}{0.49\textwidth}
	\centering
	\includegraphics[width=\textwidth]{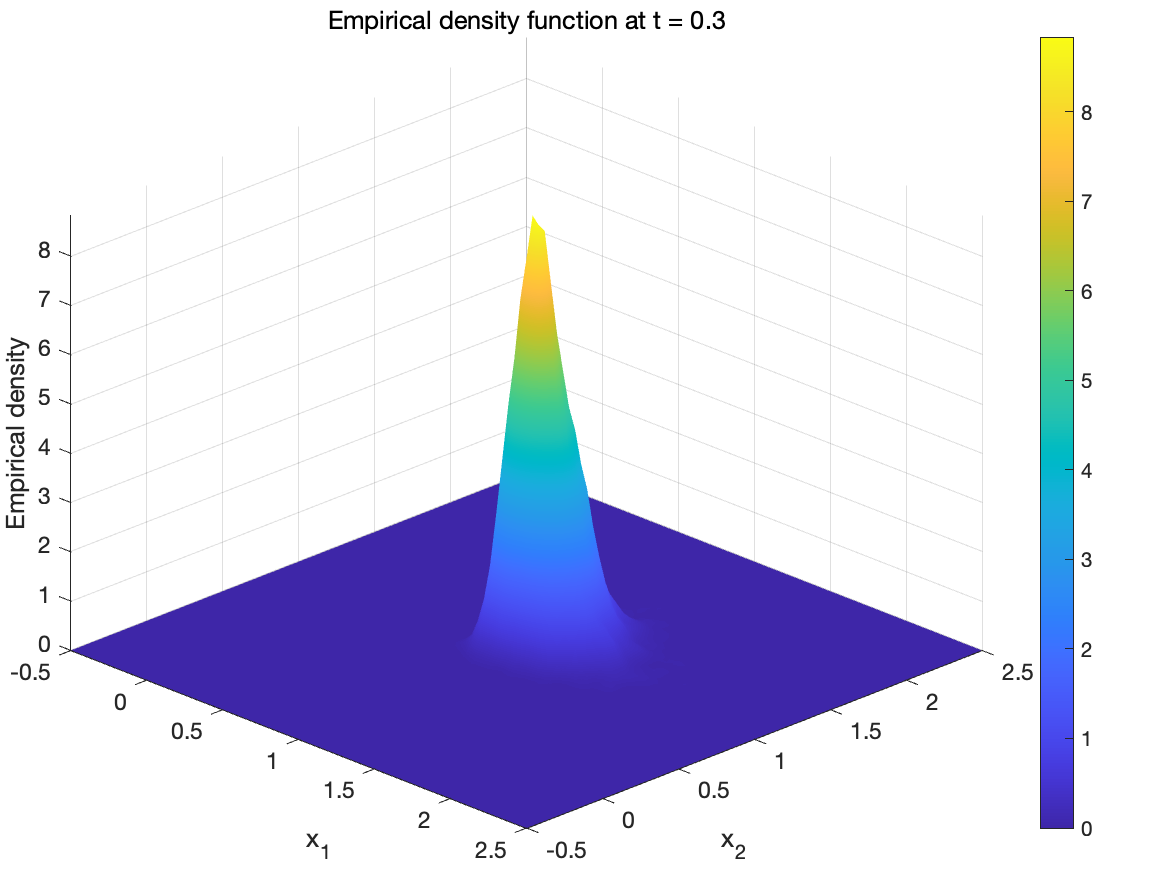}
\end{subfigure}
\hfill
\begin{subfigure}{0.49\textwidth}
	\centering
	\includegraphics[width=\textwidth]{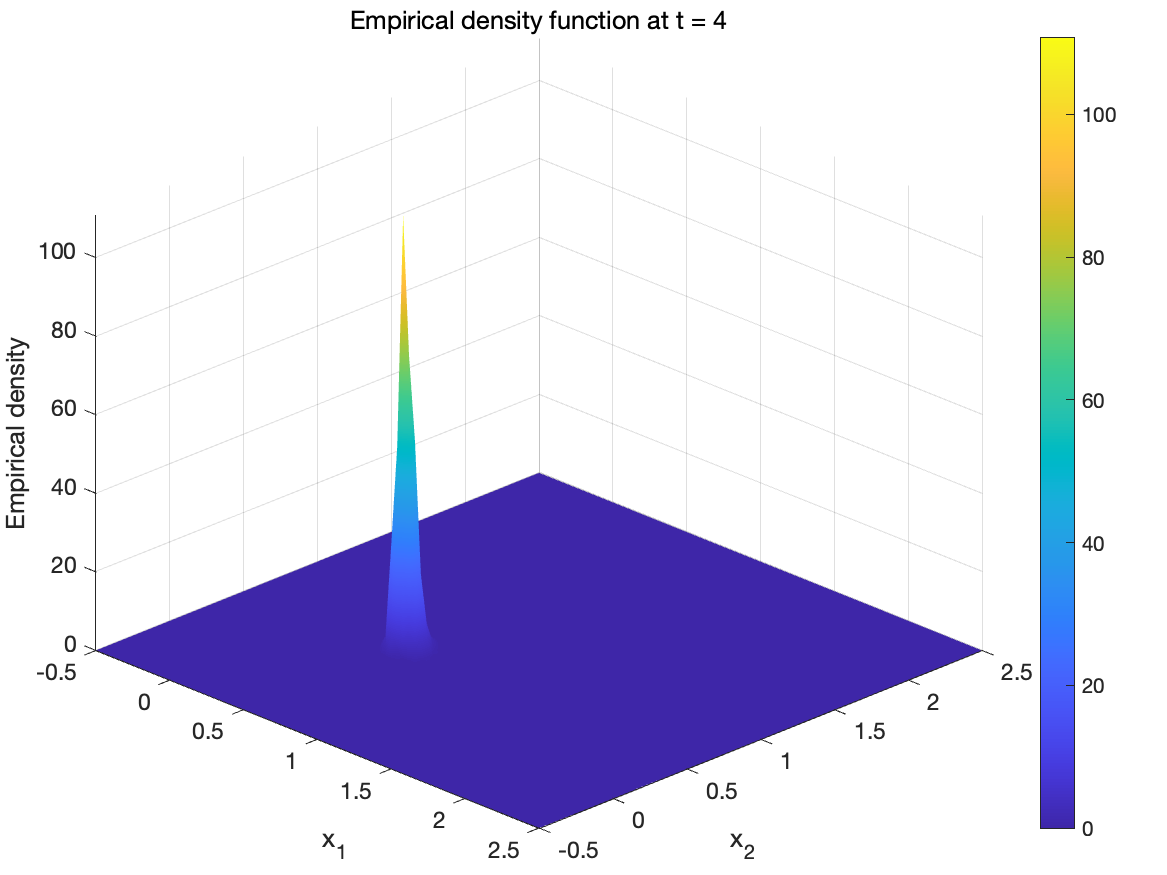}
\end{subfigure}
\hfill
\begin{subfigure}{0.49\textwidth}
	\centering
	\includegraphics[width=\textwidth]{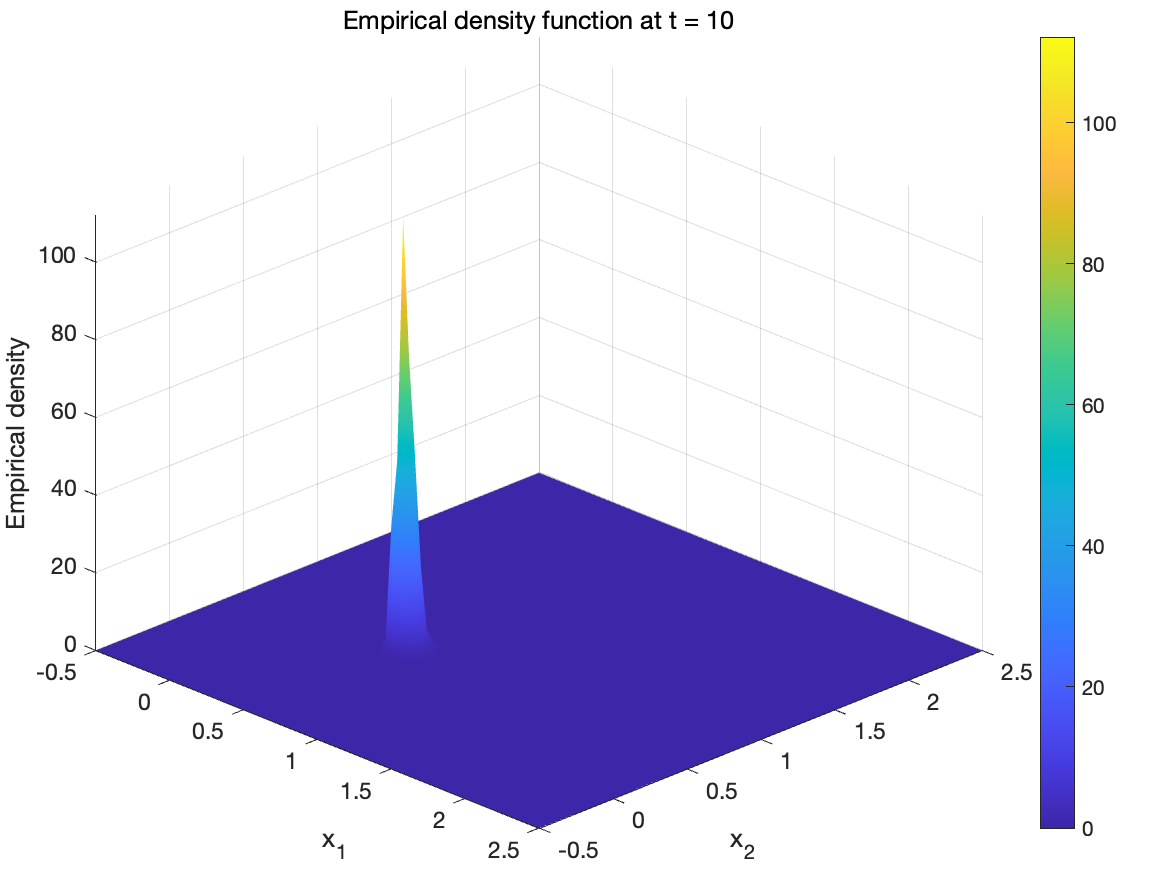}
\end{subfigure}
\caption{Empirical density function at $t=0.1$, $t=0.3$, $t=4$ and $t=10$}
\label{fig:Example2}
\end{figure}
\begin{figure}[H]
\centering
\includegraphics[width=0.7\linewidth]{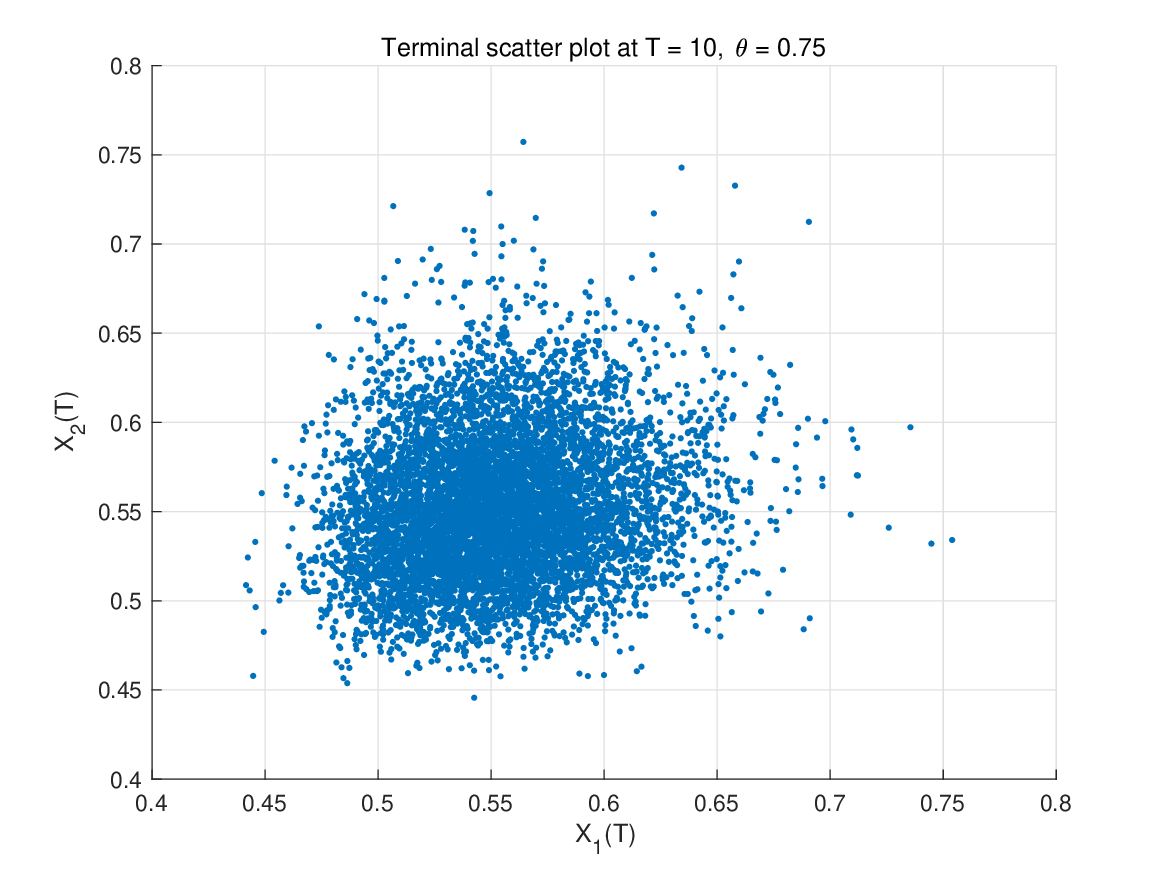}
\caption{Terminal scatter plot at $T=10$}
\label{fig:Example2-2}
\end{figure}

\section{Conclusions, discussions and future research}\label{sec:conclusion}
In this paper, the stochastic theta method is used to approximate the invariant measure of some SDEs, whose drift and diffusion coefficients are allowed to contain super-linear terms. The existence and uniqueness of the numerical invariant measure of the numerical solution generated by the stochastic theta method is proved. And the convergence of the numerical invariant measure to the exact invariant measure of the underlying SDEs is also proved.
\par
As the second moments are used, the coefficient of the linear term in the drift part has to be negative (the $-\alpha$ in Assumption \ref{assum:2.2}). Meanwhile, it is well known that if some $p$th moments are used for some small $p\in(0,1)$, one may still get the similar results \cite{LiMaYangYuan2018}. Therefore, one of the potential future research is to release the requirement in Assumption \ref{assum:2.2}.
\par
Another very promising future research is to investigate the stochastic theta method for some SDEs with state-dependent switching, as numerical invariant measures of this type of SDEs could help with sampling from finite mixture distributions, which then are quite useful in various machine learning topics \cite{Tretyakov2025}.

\bibliographystyle{abbrv}
\bibliography{myref}  

@book {Allen2007,
    AUTHOR = {Allen, E.},
     TITLE = {Modeling with {I}t\^o{} stochastic differential equations},
    SERIES = {Mathematical Modelling: Theory and Applications},
    VOLUME = {22},
 PUBLISHER = {Springer, Dordrecht},
      YEAR = {2007},
     PAGES = {xii+228},
      ISBN = {978-1-4020-5952-0},
   MRCLASS = {60-02 (60H10 60H35 65C30)},
  MRNUMBER = {2292765},
}

@book {Allen2015,
    AUTHOR = {Allen, Linda J. S.},
     TITLE = {Stochastic population and epidemic models},
    SERIES = {Mathematical Biosciences Institute Lecture Series. Stochastics
              in Biological Systems},
    VOLUME = {1.3},
      NOTE = {Persistence and extinction},
 PUBLISHER = {Springer, Cham; MBI Mathematical Biosciences Institute, Ohio
              State University, Columbus, OH},
      YEAR = {2015},
     PAGES = {x+47},
      ISBN = {978-3-319-21553-2; 978-3-319-21554-9},
   MRCLASS = {92D30},
  MRNUMBER = {3381295},
       DOI = {10.1007/978-3-319-21554-9},
}

@article {BaoShaoYuan2016,
    AUTHOR = {Bao, Jianhai and Shao, Jinghai and Yuan, Chenggui},
     TITLE = {Approximation of invariant measures for regime-switching
              diffusions},
   JOURNAL = {Potential Anal.},
  FJOURNAL = {Potential Analysis. An International Journal Devoted to the
              Interactions between Potential Theory, Probability Theory,
              Geometry and Functional Analysis},
    VOLUME = {44},
      YEAR = {2016},
    NUMBER = {4},
     PAGES = {707--727},
      ISSN = {0926-2601,1572-929X},
   MRCLASS = {60H10 (60H35 60J27 60J60)},
  MRNUMBER = {3490546},
MRREVIEWER = {Xiaoyue\ Li},
       DOI = {10.1007/s11118-015-9526-x},
}

@ARTICLE{CaoTanGaoXuChenHangLi2024,
  author={Cao, Hanqun and Tan, Cheng and Gao, Zhangyang and Xu, Yilun and Chen, Guangyong and Heng, Pheng-Ann and Li, Stan Z.},
  JOURNAL = {IEEE Trans. Knowl. Data Eng.},
  fjournal={IEEE Transactions on Knowledge and Data Engineering}, 
  title={A Survey on Generative Diffusion Models}, 
  year={2024},
  volume={36},
  number={7},
  pages={2814-2830},
  doi={10.1109/TKDE.2024.3361474}
}

@article {ChenCaoChen2025,
    AUTHOR = {Chen, Ziheng and Cao, Liangmin and Chen, Lin},
     TITLE = {Stochastic theta methods for random periodic solution of
              stochastic differential equations under non-globally
              {L}ipschitz conditions},
   JOURNAL = {Numer. Algorithms},
  FJOURNAL = {Numerical Algorithms},
    VOLUME = {99},
      YEAR = {2025},
    NUMBER = {2},
     PAGES = {651--681},
      ISSN = {1017-1398,1572-9265},
   MRCLASS = {65C30 (60H10 60H35)},
  MRNUMBER = {4903819},
MRREVIEWER = {Lihai\ Ji},
       DOI = {10.1007/s11075-024-01892-y},
}

@article {ChenLiuWu2026,
    AUTHOR = {Chen, Ziheng and Liu, Jiao and Wu, Anxin},
     TITLE = {Strong and weak convergence orders of numerical methods for
              {SDE}s driven by time-changed {L}\'evy noise},
   JOURNAL = {Appl. Numer. Math.},
  FJOURNAL = {Applied Numerical Mathematics. An IMACS Journal},
    VOLUME = {221},
      YEAR = {2026},
     PAGES = {46--62},
      ISSN = {0168-9274,1873-5460},
   MRCLASS = {65C30 (60G51 60H10 60H35)},
  MRNUMBER = {4993437},
       DOI = {10.1016/j.apnum.2025.11.007},
}

@article {Higham2000,
    AUTHOR = {Higham, Desmond J.},
     TITLE = {Mean-square and asymptotic stability of the stochastic theta
              method},
   JOURNAL = {SIAM J. Numer. Anal.},
  FJOURNAL = {SIAM Journal on Numerical Analysis},
    VOLUME = {38},
      YEAR = {2000},
    NUMBER = {3},
     PAGES = {753--769},
      ISSN = {0036-1429,1095-7170},
   MRCLASS = {60H10 (65C50)},
  MRNUMBER = {1781202},
       DOI = {10.1137/S003614299834736X},
}

@article {HighamMaoYuan2007,
    AUTHOR = {Higham, Desmond J. and Mao, Xuerong and Yuan, Chenggui},
     TITLE = {Preserving exponential mean-square stability in the simulation
              of hybrid stochastic differential equations},
   JOURNAL = {Numer. Math.},
  FJOURNAL = {Numerische Mathematik},
    VOLUME = {108},
      YEAR = {2007},
    NUMBER = {2},
     PAGES = {295--325},
      ISSN = {0029-599X,0945-3245},
   MRCLASS = {60H35 (60H10 65C30)},
  MRNUMBER = {2358006},
MRREVIEWER = {Henri\ Schurz},
       DOI = {10.1007/s00211-007-0113-y},
}

@article {Huang2014,
    AUTHOR = {Huang, Chengming},
     TITLE = {Mean square stability and dissipativity of two classes of
              theta methods for systems of stochastic delay differential
              equations},
   JOURNAL = {J. Comput. Appl. Math.},
  FJOURNAL = {Journal of Computational and Applied Mathematics},
    VOLUME = {259},
      YEAR = {2014},
     PAGES = {77--86},
      ISSN = {0377-0427,1879-1778},
   MRCLASS = {65C30},
  MRNUMBER = {3123472},
MRREVIEWER = {Gabriela\ Mircea},
       DOI = {10.1016/j.cam.2013.03.038},
}

@article {JiangWengLiu2020,
    AUTHOR = {Jiang, Yanan and Weng, Lihui and Liu, Wei},
     TITLE = {Stationary distribution of the stochastic theta method for
              nonlinear stochastic differential equations},
   JOURNAL = {Numer. Algorithms},
  FJOURNAL = {Numerical Algorithms},
    VOLUME = {83},
      YEAR = {2020},
    NUMBER = {4},
     PAGES = {1531--1553},
      ISSN = {1017-1398,1572-9265},
   MRCLASS = {65C30 (60H10)},
  MRNUMBER = {4088603},
MRREVIEWER = {Roman\ N.\ Makarov},
       DOI = {10.1007/s11075-019-00735-5},
}

@article {LiHuHuangWang2023,
    AUTHOR = {Li, Min and Hu, Yaozhong and Huang, Chengming and Wang, Xiong},
     TITLE = {Mean square stability of stochastic theta method for
              stochastic differential equations driven by fractional
              {B}rownian motion},
   JOURNAL = {J. Comput. Appl. Math.},
  FJOURNAL = {Journal of Computational and Applied Mathematics},
    VOLUME = {420},
      YEAR = {2023},
     PAGES = {Paper No. 114804, 24},
      ISSN = {0377-0427,1879-1778},
   MRCLASS = {65C30 (34F05 60G22 60H10 65L20)},
  MRNUMBER = {4488084},
MRREVIEWER = {Jialin\ Hong},
       DOI = {10.1016/j.cam.2022.114804},
}

@article {LiGan2012,
    AUTHOR = {Li, Qiyong and Gan, Siqing},
     TITLE = {Mean-square exponential stability of stochastic theta methods
              for nonlinear stochastic delay integro-differential equations},
   JOURNAL = {J. Appl. Math. Comput.},
  FJOURNAL = {Journal of Applied Mathematics and Computing},
    VOLUME = {39},
      YEAR = {2012},
    NUMBER = {1-2},
     PAGES = {69--87},
      ISSN = {1598-5865,1865-2085},
   MRCLASS = {65C30 (34K20 34K28 60H35 65L20)},
  MRNUMBER = {2914463},
       DOI = {10.1007/s12190-011-0510-3},
}

@article {LiMaYangYuan2018,
    AUTHOR = {Li, Xiaoyue and Ma, Qianlin and Yang, Hongfu and Yuan,
              Chenggui},
     TITLE = {The numerical invariant measure of stochastic differential
              equations with {M}arkovian switching},
   JOURNAL = {SIAM J. Numer. Anal.},
  FJOURNAL = {SIAM Journal on Numerical Analysis},
    VOLUME = {56},
      YEAR = {2018},
    NUMBER = {3},
     PAGES = {1435--1455},
      ISSN = {0036-1429,1095-7170},
   MRCLASS = {65C30 (34F05 60H10)},
  MRNUMBER = {3805852},
MRREVIEWER = {Victor\ B.\ Malyutin},
       DOI = {10.1137/17M1143927},
}

@article {LiMaoYin2019,
    AUTHOR = {Li, Xiaoyue and Mao, Xuerong and Yin, George},
     TITLE = {Explicit numerical approximations for stochastic differential
              equations in finite and infinite horizons: truncation methods,
              convergence in {$p$}th moment and stability},
   JOURNAL = {IMA J. Numer. Anal.},
  FJOURNAL = {IMA Journal of Numerical Analysis},
    VOLUME = {39},
      YEAR = {2019},
    NUMBER = {2},
     PAGES = {847--892},
      ISSN = {0272-4979,1464-3642},
   MRCLASS = {65C30 (60H10 60H35)},
  MRNUMBER = {3941887},
MRREVIEWER = {Roger\ Pettersson},
       DOI = {10.1093/imanum/dry015},
}

@article {LiuDengZhu2018,
    AUTHOR = {Liu, Linna and Deng, Feiqi and Zhu, Quanxin},
     TITLE = {Mean square stability of two classes of theta methods for
              numerical computation and simulation of delayed stochastic
              {H}opfield neural networks},
   JOURNAL = {J. Comput. Appl. Math.},
  FJOURNAL = {Journal of Computational and Applied Mathematics},
    VOLUME = {343},
      YEAR = {2018},
     PAGES = {428--447},
      ISSN = {0377-0427,1879-1778},
   MRCLASS = {65C30 (60H30 60H35)},
  MRNUMBER = {3813561},
MRREVIEWER = {John\ Masson\ Noble},
       DOI = {10.1016/j.cam.2018.04.018},
}

@article {LiuMaoWu2023,
    AUTHOR = {Liu, Wei and Mao, Xuerong and Wu, Yue},
     TITLE = {The backward {E}uler-{M}aruyama method for invariant measures
              of stochastic differential equations with super-linear
              coefficients},
   JOURNAL = {Appl. Numer. Math.},
  FJOURNAL = {Applied Numerical Mathematics. An IMACS Journal},
    VOLUME = {184},
      YEAR = {2023},
     PAGES = {137--150},
      ISSN = {0168-9274,1873-5460},
   MRCLASS = {65C30 (60H35)},
  MRNUMBER = {4499301},
MRREVIEWER = {Yunzhang\ Li},
       DOI = {10.1016/j.apnum.2022.09.017},
}

@article {LiuLiu2025,
    AUTHOR = {Liu, Zhihui and Liu, Zhizhou},
     TITLE = {Numerical unique ergodicity of monotone {SDE}s driven by
              nondegenerate multiplicative noise},
   JOURNAL = {J. Sci. Comput.},
  FJOURNAL = {Journal of Scientific Computing},
    VOLUME = {103},
      YEAR = {2025},
    NUMBER = {3},
     PAGES = {Paper No. 87, 21},
      ISSN = {0885-7474,1573-7691},
   MRCLASS = {65C30 (60H15 60H35 65M60)},
  MRNUMBER = {4899897},
       DOI = {10.1007/s10915-025-02902-4},
}

@book{Mao2007,
  title={Stochastic {D}ifferential {E}quations and {A}pplications},
  author={Mao, Xuerong},
  year={2007},
  edition={2nd},
  publisher={Elsevier}
}

@article{NiuWeiYinZeng2025,
    author = {Niu, Yuanling and Wei, Jiaxin and Yin, Zhi and Zeng, Dan},
    title = {Stochastic theta methods for free stochastic differential equations},
    journal = {IMA J. Numer. Anal.},
   FJOURNAL = {IMA Journal of Numerical Analysis},
    VOLUME = {online first},
    year = {2025},
    issn = {0272-4979},
    doi = {10.1093/imanum/draf044},
}

@article {PangWangWu2024,
    AUTHOR = {Pang, Chenxu and Wang, Xiaojie and Wu, Yue},
     TITLE = {Linear implicit approximations of invariant measures of
              semi-linear {SDE}s with non-globally {L}ipschitz coefficients},
   JOURNAL = {J. Complexity},
  FJOURNAL = {Journal of Complexity},
    VOLUME = {83},
      YEAR = {2024},
     PAGES = {Paper No. 101842, 45},
      ISSN = {0885-064X,1090-2708},
   MRCLASS = {60H35 (37M25 65C30)},
  MRNUMBER = {4721566},
       DOI = {10.1016/j.jco.2024.101842},
}

@inproceedings{Songetal2021,
  title={Score-based generative modeling through stochastic differential equations},
  author={Song, Yang and Sohl-Dickstein, Jascha and Kingma, Durk P and Kumar, Abhishek and Ermon, Stefano and Poole, Ben},
  booktitle={International Conference on Learning Representations},
  year={2021}
}

@article{Tretyakov2025,
	AUTHOR = {Tretyakov, Michael V.},
	TITLE = {Sampling from mixture distributions based on regime-switching diffusions},
    JOURNAL = {SIAM J. Sci. Comput.},
	FJOURNAL = {SIAM Journal on Scientific Computing},
	VOLUME = {47},
    YEAR = {2025},
	NUMBER = {3},
	PAGES = {A1681--A1701},
	DOI = {10.1137/24M1677113},
}

@article {WangWuDong2020,
    AUTHOR = {Wang, Xiaojie and Wu, Jiayi and Dong, Bozhang},
     TITLE = {Mean-square convergence rates of stochastic theta methods for
              {SDE}s under a coupled monotonicity condition},
   JOURNAL = {BIT},
  FJOURNAL = {BIT. Numerical Mathematics},
    VOLUME = {60},
      YEAR = {2020},
    NUMBER = {3},
     PAGES = {759--790},
      ISSN = {0006-3835,1572-9125},
   MRCLASS = {65C30 (60H10 60H35)},
  MRNUMBER = {4132905},
MRREVIEWER = {Victor\ B.\ Malyutin},
       DOI = {10.1007/s10543-019-00793-0},
}

@article {WangChenNiuNiu2023,
    AUTHOR = {Wang, Yiling and Chen, Ziheng and Niu, Mengyao and Niu,
              Yuanling},
     TITLE = {Mean-square convergence and stability of compensated
              stochastic theta methods for jump-diffusion {SDE}s with
              super-linearly growing coefficients},
   JOURNAL = {Calcolo},
  FJOURNAL = {Calcolo. A Quarterly on Numerical Analysis and Theory of
              Computation},
    VOLUME = {60},
      YEAR = {2023},
    NUMBER = {2},
     PAGES = {Paper No. 30, 32},
      ISSN = {0008-0624,1126-5434},
   MRCLASS = {60H35 (65C20 65C30)},
  MRNUMBER = {4589707},
       DOI = {10.1007/s10092-023-00524-6},
}

@article {YuanZhu2026,
    AUTHOR = {Yuan, Haiyan and Zhu, Quanxin},
     TITLE = {Efficient stability-preserving schemes for stochastic
              {M}c{K}ean-{V}lasov equations with uncertainty},
   JOURNAL = {J. Comput. Appl. Math.},
  FJOURNAL = {Journal of Computational and Applied Mathematics},
    VOLUME = {483},
      YEAR = {2026},
     PAGES = {Paper No. 117361, 22},
      ISSN = {0377-0427,1879-1778},
   MRCLASS = {65C30 (60H10)},
  MRNUMBER = {5020993},
       DOI = {10.1016/j.cam.2026.117361},
}

\end{document}